\documentclass{article} 
\usepackage{amsmath}
\usepackage{amssymb}
\usepackage{amsthm}
\usepackage{algorithm}
\usepackage{algorithmic}
\usepackage{multirow}
\usepackage{graphicx}
\usepackage{color}
\usepackage{lscape}
\usepackage{url}
\def\@themcountersep{}
\newtheorem{theorem}{Theorem}

\newtheorem{corollary}[theorem]{Corollary}

\newtheorem{lemma}[theorem]{Lemma}
\newtheorem{proposition}[theorem]{Proposition}
\newtheorem{remark}[theorem]{Remark}

\def\ba{\begin{array}}
\def\ea{\end{array}}
\def\beq{\begin{equation}}
\def\eeq{\end{equation}}
\def\beann{\begin{eqnarray*} }
\def\eeann{\end{eqnarray*}}
\def\bc{\begin{center}}
\def\ec{\end{center}}
\def\bea{\begin{eqnarray}}
\def\eea{\end{eqnarray}}

\begin{document}

\title{
Evaluating approximations of the semidefinite cone with trace normalized distance 
\thanks{This research was supported by the Japan Society for the Promotion of Science through a Grant-in-Aid for Scientific Research ((B)19H02373) from the Ministry of Education, Culture, Sports, Science and Technology of Japan.}
}

\author{Yuzhu Wang\thanks{
Graduate School of Systems and Information Engineering, University of Tsukuba, Tsukuba, Ibaraki 305-8573, Japan. email: wangyuzhu820@yahoo.co.jp
} 
 and 
Akiko Yoshise\thanks{Corresponding author. Faculty of Engineering, Information and Systems, University of Tsukuba, Tsukuba, Ibaraki 305-8573, Japan. email: yoshise@sk.tsukuba.ac.jp
}     
}

\date{May 2021 \\ Revised July 2022}

\maketitle

\begin{abstract}
We evaluate the dual cone of the set of diagonally dominant matrices (resp., scaled diagonally dominant matrices), namely ${\cal DD}_n^*$ (resp., ${\cal SDD}_n^*$), as an approximation of the semidefinite cone. We prove that the norm normalized distance, proposed by Blekherman et al. (2022), between a set ${\cal S}$ and the semidefinite cone has the same value whenever ${\cal SDD}_n^* \subseteq {\cal S} \subseteq {\cal DD}_n^*$. This implies that the norm normalized distance is not a sufficient measure to evaluate these approximations. As a new measure to compensate for the weakness of that distance, we propose a new distance, called the trace normalized distance. We prove that the trace normalized distance between ${\cal DD}_n^*$ and ${\cal S}^n_+$ has a different value from the one between ${\cal SDD}_n^*$ and ${\cal S}^n_+$ and give the exact values of these distances.

\end{abstract}

\section{Introduction}
Semidefinite programming (SDP) can provide powerful convex relaxations for combinatorial and nonconvex optimization problems (\cite{Goemans:1995:IAA:227683.227684}, \cite{Lasserre2001}, \cite{todd2001}, \cite{wolkowicz2012handbook}, etc.). SDP has the following standard form:
\begin{align*}
	\min&\ \langle C,X\rangle\\
	{\rm s.t.}&\ \langle A_j,X\rangle=b_j,j=1,2,\ldots,m,\\
	&\ X\in{\cal S}^n_+,
\end{align*}
where $C\in\mathbb{S}^n$, $A_j\in\mathbb{S}^n$, $b_j\in\mathbb{R}$ ($j=1,2,\ldots,m$), and $\langle A,B\rangle:=\sum_{i,j=1}^nA_{i,j}B_{i,j}$ is the inner product over $\mathbb{S}^n$. The space of symmetric matrices is denoted as ${\mathbb S}^n:=\{X\in\mathbb{R}^{n\times n}\mid X_{i,j}=X_{j,i} \ (1 \leq i < j \leq n) \}$ and the semidefinite cone is defined as ${\cal S}^n_+:=\{X\in\mathbb{S}^n\mid d^TXd\ge0 \ \mbox{for any} \ d\in\mathbb{R}^n \}$. 

SDP is theoretically attractive because it can be solved in polynomial time to any desired precision. However, it is difficult to solve large-scale SDP instances even using state-of-the-art solvers, such as Mosek \cite{mosek}. One technique to overcome this deficiency is to relax the semidefinite constraint and solve the resulting easier problem by using, e.g., linear programming (LP) or second order cone programming (SOCP). A typical relaxation technique is to replace the constraint $X \in \mathcal{S}_+^n$ with a relaxed constraint $X \in \mathcal{S}$, where $\mathcal{S}$ is a subset of $\mathbb{S}^n$ containing $\mathcal{S}_+^n$, i.e., $\mathcal{S}_+^n \subseteq \mathcal{S} \subseteq \mathbb{S}^n$. If $X \in \mathcal{S}$ is described by linear constraints, the resulting problem becomes an LP problem and if $X \in \mathcal{S}$ is described by second-order constraints, the resulting problem becomes an SOCP problem. Such a set ${\cal S}$ is called an outer approximation of ${\cal S}^n_+$. There are two kinds of approximation, i.e., inner approximation and outer approximation. An inner approximation (outer approximation) of ${\cal S}^n_+$ can be obtained by constructing the dual cone of an outer approximation (inner approximation). We will focus on the outer approximations of ${\cal S}^n_+$ and refer to them as approximations of ${\cal S}^n_+$ throughout this paper.

Several sets have been used as approximations of the semidefinite cone, including the $k$-PSD closure, namely ${\cal S}^{n,k}$ (\cite{blekherman2020}) and the dual cone of the set of diagonally dominant matrices (resp., scaled diagonally dominant matrices), namely ${\cal DD}_n^*$ (resp., ${\cal SDD}_n^*$) (\cite{ahmadi2017dsos}). Multiple experiments have shown the efficiency of cutting-plane methods using these approximations (\cite{ahmadi2017dsos}, \cite{bertsimas2020}, \cite{wang2021}). Although the inclusive relationship of these approximations has been given (see, e.g. \cite{ahmadi2017dsos}, \cite{bertsimas2020}), theoretical analyses of how well these sets approximate the semidefinite cone have been limited.

{Fawzi \cite{doi:10.1287/moor.2020.1077} evaluated how polytopes can approximate a compact slice of the semidefinite cone by using a measure called extension complexity.} Bertsimas and Cory-Wright \cite{bertsimas2020} evaluated ${\cal DD}_n^*$ and ${\cal SDD}_n^*$ as approximations of the semidefinite cone by comparing the lower bounds on the minimum eigenvalues of matrices from these two sets. Blekherman et al. \cite{blekherman2020} evaluated ${\cal S}^{n,k}$, as an approximation of the semidefinite cone, by using an evaluation method called the norm normalized distance. The norm normalized distance between a given approximation ${\cal S} \subseteq \mathbb{S}^n$ and ${\cal S}^n_+$ is the {maximum Frobenius distance} from a matrix $X\in{\cal S}$ to ${\cal S}^n_+$, where the Frobenius norm of the matrix $X$ is assumed to be one. They obtained several upper bounds and lower bounds on the norm normalized distance between ${\cal S}^{n,k}$ and ${\cal S}^n_+$.

In this paper, we first show that the norm normalized distance between a set ${\cal S}$ and ${\cal S}^n_+$ has the same value whenever ${\cal SDD}_n^* \subseteq {\cal S} \subseteq {\cal DD}_n^*$. This implies that the norm normalized distance is not a sufficient measure to {differentiate these approximations}. As a new measure to compensate for the weakness of that distance, we introduce a new distance, called the trace normalized distance. We prove that the trace normalized distance between ${\cal DD}_n^*$ and ${\cal S}^n_+$ has a different value from the one between ${\cal SDD}_n^*$ and ${\cal S}^n_+$ and give the exact values of these distances.

The organization of this paper is as follows. Section \ref{sec:1} introduces approximations of the semidefinite cone ${\cal S}^{n,k}$, ${\cal DD}_n^*$ and ${\cal SDD}_n^*$ and their inclusive relationship. In Section \ref{sec:a2}, the norm normalized distance proposed by Blekherman et al. \cite{blekherman2020} is used to evaluate ${\cal DD}_n^*$ and ${\cal SDD}_n^*$. In Section \ref{sec:2}, the trace normalized distance is proposed and used to evaluate ${\cal DD}_n^*$ and ${\cal SDD}_n^*$. We conclude our work in Section \ref{sec:3}.

\section{Approximations of the semidefinite cone}
\label{sec:1}
Consider the following three sets as approximations of the semidefinite cone. Let $k$ and $n$ be positive integers satisfying $2\leq k\leq n$ and
\begin{align}
	{\cal S}^{n,k}:=&\{X\in\mathbb{S}^n\mid\mbox{ All } k\times k\mbox{ principal submatrices of }X \mbox{ are positive semidefinite} \}.\label{def:snk}\\
	{\cal DD}_n^*:=&\{X\in\mathbb{S}^n\mid X_{i,i}+X_{j,j}\pm 2X_{i,j}\ge0\ (1\leq i\leq j\leq n)\}.\label{def:dd}\\
	{\cal SDD}_n^*:=&\{X\in\mathbb{S}^n\mid X_{i,i}X_{j,j}\ge X_{i,j}^2\ (1\leq i<j\leq n),X_{i,i}\ge0\ (i=1,\ldots,n)\}.\label{def:sdd}
\end{align}

${\cal S}^{n,k}$ is called the $k$-$PSD\ closure$, whose properties are discussed in \cite{blekherman2020}. It is obvious from the definition $(\ref{def:snk})$ that ${\cal S}^n_+={\cal S}^{n,n}\subseteq{\cal S}^{n,k_1}\subseteq{\cal S}^{n,k_2}$ when $n\ge k_1\ge k_2\ge2$. 

${\cal DD}_n^*$ (resp. ${\cal SDD}_n^*$) is the dual cone of the set of diagonally dominant matrices (resp. scaled diagonally dominant matrices). These duality relationships imply that ${\cal SDD}_n^*$ is a subset of ${\cal DD}_n^*$. It is worth noting that ${\cal DD}_n^*$ and ${\cal SDD}_n^*$ are used as approximations of the semidefinite cone in cutting-plane methods (\cite{ahmadi2017optimization},\cite{ahmadi2017dsos},\cite{bertsimas2020},\cite{wang2021}) and facial reduction methods \cite{permenter2014partial}.

According to definitions $(\ref{def:snk})$ and $(\ref{def:sdd})$, it is clear that ${\cal S}^{n,2}={\cal SDD}_n^*$. The relations among ${\cal S}^{n,k}$, ${\cal DD}_n^*$ and ${\cal SDD}_n^*$ can be concluded to be 
\begin{align}
	{\cal S}^n_+={\cal S}^{n,n}\subseteq\cdots\subseteq{\cal S}^{n,2}={\cal SDD}_n^*\subseteq{\cal DD}_n^*.\label{setrela}
\end{align}

\section{The norm normalized distance between the semidefinite cone and its approximation}
\label{sec:a2}
Blekherman et al. \cite{blekherman2020} proposed a method of evaluating approximations of the semidefinite cone, which is based on the maximum distance from a matrix in a given approximation $\mathcal{S}_+^n\subseteq {\cal S} \subseteq \mathbb{S}^n$ to $\mathcal{S}_+^n$. A feature of their method is that the distance is evaluated under the constraint that the value of the Frobenius norm is one. The norm normalized distance between a set ${\cal S}$ and ${\cal S}^n_+$ is defined as 
\begin{align}
	\overline{\mbox{dist}}_F({\cal S},{\cal S}^n_+):=&\sup_{X\in{\cal S},\|X\|_F=1}\|X-{\rm P}_{{\cal S}^n_+}(X)\|_F,\label{maxdis}
\end{align}
where ${\rm P}_{{\cal S}^n_+}(X):={\rm argmin}_{Y\in{\cal S}^n_+}\|X-Y\|_F$ is the metric projection of $X$ on ${\cal S}^n_+$. 

The authors of \cite{blekherman2020} showed that $\overline{\mbox{dist}}_F({\cal S}^{n,k},{\cal S}^n_+) \leq \frac{n-k}{n+k-2}$. Through a similar discussion to the one given there, we can prove the following theorem:

\setcounter{theorem}{0}
\begin{theorem}\label{thmy2}
	For $n\ge4$, 
	\begin{align*}
		\overline{\rm dist}_F({\cal DD}_n^*,{\cal S}^n_+)=\overline{\rm dist}_F({\cal SDD}_n^*,{\cal S}^n_+)=\frac{n-2}{n}.
	\end{align*}
\end{theorem}

Note that the second equality (i.e., $\overline{\rm dist}_F({\cal SDD}_n^*,{\cal S}^n_+)=\frac{n-2}{n}$) is a corollary of Blekherman et al.'s lower and upper bounds on ${\cal S}^{n,k}$ (Theorem 1 and 3 of \cite{blekherman2020}). {The proof for $\overline{\rm dist}_F({\cal DD}_n^*,{\cal S}^n_+)$} is provided in Appendix \ref{app1}.

Theorem \ref{thmy2} shows, unfortunately, that the norm normalized distance $(\ref{maxdis})$ gives the same value $\overline{\mbox{dist}}_F({\cal S},{\cal S}^n_+)= \frac{n-2}{n}$ for any approximation ${\cal S} \subseteq \mathbb{S}^n$ whenever it satisfies ${\cal SDD}_n^* \subseteq {\cal S}\subseteq {\cal DD}_n^*$. In the next section, we introduce a new distance, called the trace normalized distance. We show that the new distance between ${\cal SDD}_n^*$ and ${\cal S}^n_+$ has a different value from the one between ${\cal DD}_n^*$ and ${\cal S}^n_+$.
\section{The trace normalized distance between the semidefinite cone and its approximation}
\label{sec:2}

The Frobenius norm normalized distance can be generalized by expanding the normalization method and the distance function. {For example, let ${\cal S}$ be an approximation such that ${\cal S}^n_+\subseteq{\cal S}\subseteq\mathbb{S}^n$, $p\in[1,\infty]$, and $f:\mathbb{S}^n\rightarrow\mathbb{R}$ be a normalization function which requires the set $\{X\in{\cal S}\mid f(X)=1\}$ to be bounded.} One can define the $(f(\cdot),p)$ distance from ${\cal S}$ to ${\cal S}^n_+$ as:	
\begin{align}
	\overline{\mbox{dist}}_{(f(\cdot),p)}({\cal S},{\cal S}^n_+):=&\sup_{X\in{\cal S},f(X)=1}\|X-{\rm P}_{{\cal S}^n_+}(X)\|_p,\label{distance}
\end{align}
where $\|X\|_p:=\sqrt[p]{\sum_{i=1}^n|\lambda_i|^p}$, where $\lambda_1,\ldots,\lambda_n$ are the eigenvalues of $X$, denotes the Schatten $p$-norm of $X\in\mathbb{S}^n$. {Note that $\|X\|_\infty:=\max\{|\lambda_1|,\ldots,|\lambda_n|\}$.} In this notation, the Frobenius norm normalized distance (\ref{maxdis}) can be rewritten as an $(\|\cdot\|_F,2)$ distance (\ref{distance}). 

Recently, Blekherman et al. \cite{blekherman2021} studied a hyperbolic relaxation of ${\cal S}^{n,k}$ and provided an upper bound on the $(f(\cdot),\infty)$ distance from ${\cal S}^{n,k}$ to ${\cal S}^n_+$, where $f$ can be any unitarily invariant matrix norm or the trace function. {It can be observed that the infinity distance of a matrix $X$ to ${\cal S}^n_+$, i.e., $\|X-{\rm P}_{{\cal S}^n_+}(X)\|_\infty$, is the absolute value of the most negative eigenvalue of $X$.} In this paper, instead of evaluating the most negative eigenvalue of the matrices in the approximation ${\cal S}$, we try to figure out how ${\cal S}$ approximates ${\cal S}^n_+$ from a geometric point of view; i.e., we set $p=2$ and stick with the Euclidean norm $\|\cdot\|_F$ of $\mathbb{S}^n$.

One reason why the Frobenius normalized distance (\ref{maxdis}) fails to distinguish ${\cal DD}^*_n$ and ${\cal SDD}^*_n$ might be that the constraint $\|X\|_F=1$ is restrictive and makes the set $\{X\in\mathbb{S}^n\mid\|X\|_F=1\}$ bounded. The required properties for normalization methods (e.g., $f(\cdot)$) here are only to make $\{X\in{\cal DD}^*_n\mid f(X)=1\}$ and $\{X\in{\cal SDD}^*_n\mid f(X)=1\}$ bounded.

There are some choices of the normalization method. For example, one may consider bounding another norm (i.e., $\|X\|=1$), the determinant (i.e., $|X|=1$), or the trace (i.e., ${\rm Tr}(X):=\sum_{i=1}^nX_{i,i}= 1$) of all matrices $X\in\mathbb{S}^n$. In the case of using other norms, we know from the equivalence of norms (e.g., Corollary 5.4.5 \cite{horn1990matrix}) that for any norm $\|\cdot\|$ on $\mathbb{S}^n$ and for every matrix $X\in\mathbb{S}^n$, if $\|X\|=1$, then $\|X\|_F$ is bounded from above.  This shows that the set $\{X\in\mathbb{S}^n\mid\|X\|=1\}$ is also bounded. {Although it is not known whether distances defined with other norms can successfully distinguish ${\cal DD}_n^*$ and ${\cal SDD}_n^*$, studying these distances might be challenging, since the norm (e.g., the Schatten $p$-norm with $p\neq 2$) would be a complicated function in terms of entries of the matrix.} As for the choice of determinant, although $\{X\in\mathbb{S}^n\mid |X|=1\}$ is unbounded, one may notice that $\{X\in{\cal DD}_n^*\mid |X|=1\}$ is also unbounded, which shows that the determinant is unavailable as a normalization method. Similarly, $\{X\in\mathbb{S}^n\mid {\rm Tr}(X)=1\}$ is unbounded, but one can show that $\{X\in{\cal DD}_n^*\mid {\rm Tr}(X)=1\}$ and $\{X\in{\cal SDD}_n^*\mid {\rm Tr}(X)=1\}$ are bounded. Note as well that since Tr$(X)=1$ is a linear constraint, the subset of the polyhedral cone ${\cal DD}_n^*$ with trace equal to $1$, i.e., $\{X\in{\cal DD}_n^*\mid {\rm Tr}(X)=1\}$, is still polyhedral. In fact, we used this fact to derive the trace normalized distance between ${\cal DD}^*_n$ and ${\cal S}^n_+$. From the above discussion, we consider that a distance using the trace is effective for identifying the sets ${\cal DD}^*_n$ and ${\cal SDD}^*_n$. 

\setcounter{theorem}{0}
\begin{remark}
	There actually is a very interesting norm that can be regarded as a normalization method equivalent to Tr$(\cdot)$ on ${\cal DD}_n^*$ and ${\cal SDD}^*_n$. Let ${\cal K}$ be a regular cone (i.e., ${\cal K}$ is convex closed pointed with nonempty interior) where $I\in{\rm int}K^*$ and $\|X\|_I:=\min\{{\rm Tr}(X_1+X_2)\mid X_1-X_2=X,\ X_1,X_2\in{\cal K}\}$ be the norm induced by $I$, which was introduced in \cite{freund2006}. Proposition 1 of \cite{freund2006} implies that $\|X\|_I={\rm Tr}(X)$ if $X\in{\cal K}$. By letting ${\cal K}={\cal DD}^*_n$, it is straightforward to see that $\{X\in{\cal DD}^*_n\mid \|X\|_I=1\}=\{X\in{\cal DD}^*_n\mid {\rm Tr}(X)=1\}$. Since ${\cal SDD}^*_n\subseteq{\cal DD}^*_n$, $\|\cdot\|_I$ and Tr$(\cdot)$ are also equivalent on ${\cal SDD}^*_n$.
\end{remark}

We are now ready to use the Frobenius norm as the distance function and the trace function as the normalization method in (\ref{distance}). To simplify the notation, we will rewrite the $({\rm Tr}(\cdot),2)$ distance (\ref{distance}) into the following trace normalized distance from a set ${\cal S}$ and ${\cal S}^n_+$:
\begin{align}
	\overline{\mbox{dist}}_T({\cal S},{\cal S}^n_+):=\sup_{X\in{\cal S},{\rm Tr}(X)=1}\|X-{\rm P}_{{\cal S}^n_+}(X)\|_F,\label{deftrdis}
\end{align}

As will be shown in the sections below, $\overline{{\rm dist}}_T({\cal SDD}_n^*,{\cal S}^n_+)$ and $\overline{{\rm dist}}_T({\cal DD}_n^*,{\cal S}^n_+)$ are different, i.e., $\overline{{\rm dist}}_T({\cal SDD}_n^*,{\cal S}^n_+)= \frac{n-2}{n}$ (Theorem \ref{thsddt}) and $\overline{{\rm dist}}_T({\cal DD}_n^*,{\cal S}^n_+)=\frac{\sqrt{n}-1}{2}$ (Theorem \ref{thmy7}). 

Table \ref{tab1} compares the proof techniques used by Blekherman et al. \cite{blekherman2020} and those in this paper.{ As can be seen in Table \ref{tab1}, the results for $\overline{\mbox{dist}}_{F}({\cal SDD}^*_n,{\cal S}^n_+)$, $\overline{\mbox{dist}}_{F}({\cal DD}^*_n,{\cal S}^n_+)$ and $\overline{\mbox{dist}}_{T}({\cal SDD}^*_n,{\cal S}^n_+)$ in this paper are proved by using techniques presented by Blekherman et al. \cite{blekherman2020}. Specifically, the "averaging technique" in Theorem 1 \cite{blekherman2020} (resp., the "specific matrix construction technique" in Theorem 3 \cite{blekherman2020}) is modified and used to prove Theorem 1 and Lemma 2 (resp., Lemma 1) in this paper. As for Theorem 3 in this paper, we obtained the exact result of $\overline{\mbox{dist}}_{T}({\cal DD}^*_n,{\cal S}^n_+)$ by analyzing the structures of extreme points of $\{X\in{\cal DD}_n^*\mid {\rm Tr}(X)=1\}$ and by adopting the Bauer maximum principle. This new approach benefits from the polyhedral structure of $\{X\in{\cal DD}_n^*\mid {\rm Tr}(X)=1\}$.
	
	Note that in the proof of Theorem 2 \cite{blekherman2020}, Blekherman et al. used the Cauchy interlacing theorem to calculate the number of negative eigenvalues of an arbitrary matrix $X\in{\cal S}^{n,k}$, which is at most $n-k$. They showed that $\|X-P_{{\cal S}^n_+}(X)\|\leq \sqrt{(n-k)\lambda_1^2}$, where $\lambda_1$ is the most negative eigenvalue of $X$. Finally, to get an upper bound of $\overline{\mbox{dist}}_{F}({\cal S}^{n,k},{\cal S}^n_+)$, they only need to bound $\lambda_1$. However, in the setting of this paper (e.g., to analyze ${\cal SDD}^*_n$ and ${\cal DD}^*_n$), this approach may not be efficient. For example, because ${\cal SDD}^*_n={\cal S}^{n,2}$, the matrices in ${\cal SDD}^*_n$ could have at most $n-2$ negative eigenvalues. If we bound the most negative eigenvalue of a matrix in ${\cal SDD}^*_n$ and use the Cauchy interlace theorem, we would inflate the bound by $\sqrt{n-2}$, making the upper bound of $\overline{\mbox{dist}}_{F}({\cal SDD}^*_n,{\cal S}^n_+)$ unnecessarily large.
}

\begin{table}[h]
	\centering
	
	\begin{tabular}{cc|l}
		\hline
		\multicolumn{2}{c|}{\begin{tabular}[c]{@{}c@{}}Proof techniques\\ of each theorem\end{tabular}}                   & \multicolumn{1}{c}{Blekherman et al. 2022 \cite{blekherman2020}}             \\ \hline
		\multicolumn{2}{c|}{Object}                                                                                           & \multicolumn{1}{c}{${\cal S}^{n,k}$}                   \\ \hline
		\multicolumn{1}{c|}{\multirow{4}{*}{$\overline{\mbox{dist}}_{F}(\cdot,{\cal S}^n_+)$}} & \multirow{2}{*}{Upper bound} & (Theorem 1 {\cite{blekherman2020}}) Averaging technique                        \\ \cline{3-3} 
		\multicolumn{1}{c|}{}                                                                  &                              & (Theorem 2 {\cite{blekherman2020}}) Bound the most negative eigenvalue         \\ \cline{2-3} 
		\multicolumn{1}{c|}{}                                                                  & \multirow{2}{*}{Lower bound} & (Theorem 3 {\cite{blekherman2020}}) Construct a matrix far from ${\cal S}^n_+$ \\ \cline{3-3} 
		\multicolumn{1}{c|}{}                                                                  &                              & (Theorem 4 {\cite{blekherman2020}}) Restricted Isometry Property               \\ \hline
	\end{tabular}
	
	\vspace{3mm}
	
	\begin{tabular}{cc|ll}
		\hline
		\multicolumn{2}{c|}{\begin{tabular}[c]{@{}c@{}}Proof techniques\\ of each theorem\end{tabular}}                   & \multicolumn{2}{c}{This paper}                                                                                                                                                                                                                                                                                                                                                                            \\ \hline
		\multicolumn{2}{c|}{Object}                                                                                           & \multicolumn{1}{c|}{${\cal SDD}^*_n$}                                                                                                                                                                          & \multicolumn{1}{c}{${\cal DD}^*_n$}                                                                                                                                                      \\ \hline
		\multicolumn{1}{c|}{\multirow{4}{*}{$\overline{\mbox{dist}}_{F}(\cdot,{\cal S}^n_+)$}} & \multirow{2}{*}{Upper bound} & \multicolumn{1}{l|}{\multirow{2}{*}{\begin{tabular}[c]{@{}l@{}}(Theorem 1) \\ $\overline{\mbox{dist}}_{F}({\cal SDD}^*_n,{\cal S}^n_+)\leq\overline{\mbox{dist}}_F({\cal DD}^*_n,{\cal S}^n_+)$\end{tabular}}} & \multirow{2}{*}{\begin{tabular}[c]{@{}l@{}}(Theorem 1) \\ Averaging technique\end{tabular}}                                                                                              \\
		\multicolumn{1}{c|}{}                                                                  &                              & \multicolumn{1}{l|}{}                                                                                                                                                                                          &                                                                                                                                                                                          \\ \cline{2-4} 
		\multicolumn{1}{c|}{}                                                                  & \multirow{2}{*}{Lower bound} & \multicolumn{1}{l|}{\multirow{2}{*}{\begin{tabular}[c]{@{}l@{}}(Theorem 1)\\ Corollary of Theorem 3 \cite{blekherman2020}\end{tabular}}}                                                                            & \multirow{2}{*}{\begin{tabular}[c]{@{}l@{}}(Theorem 1) \\ $\overline{\mbox{dist}}_{F}({\cal DD}^*_n,{\cal S}^n_+)\ge\overline{\mbox{dist}}_F({\cal SDD}^*_n,{\cal S}^n_+)$\end{tabular}} \\
		\multicolumn{1}{c|}{}                                                                  &                              & \multicolumn{1}{l|}{}                                                                                                                                                                                          &                                                                                                                                                                                          \\ \hline
		\multicolumn{1}{c|}{\multirow{2}{*}{$\overline{\mbox{dist}}_{T}(\cdot,{\cal S}^n_+)$}} & Upper bound                  & \multicolumn{1}{l|}{\begin{tabular}[c]{@{}l@{}}(Lemma 2) \\ Averaging technique\end{tabular}}                                                                                                                  & \multirow{2}{*}{\begin{tabular}[c]{@{}l@{}}(Theorem 3)\\ Analyze extreme points of {$\{X\in{\cal DD}_n^*\mid {\rm Tr}(X)=1\}$} \\ and use the Bauer maximum principle\end{tabular}}          \\ \cline{2-3}
		\multicolumn{1}{c|}{}                                                                  & Lower bound                  & \multicolumn{1}{l|}{\begin{tabular}[c]{@{}l@{}}(Lemma 1) \\ Construct a matrix far from ${\cal S}^n_+$\end{tabular}}                                                                                           &                                                                                                                                                                                          \\ \hline
	\end{tabular}
	\caption{Comparison of proof techniques in Blekherman et al. \cite{blekherman2020} and those in this paper.}
	\label{tab1}
	
\end{table}

\subsection{The trace normalized distance between ${\cal SDD}_n^*$ and ${\cal S}^n_+$}

\begin{theorem}\label{thsddt}
	For all $n\ge 2$, 
	\begin{align*}
		\overline{{\rm dist}}_T({\cal SDD}_n^*,{\cal S}^n_+)= \frac{n-2}{n}.
	\end{align*}
\end{theorem}

To prove this theorem, we need Lemmas \ref{thsddlow} and \ref{thmy3}. Lemma \ref{thsddlow} gives a lower bound on $\overline{{\rm dist}}_T({\cal SDD}_n^*,{\cal S}^n_+)$ and Lemma \ref{thmy3} gives an upper bound on $\overline{{\rm dist}}_T({\cal SDD}_n^*,{\cal S}^n_+)$. Here, we assume that $n\ge3$. If $n=2${, from the fact that ${\cal SDD}_2={\cal S}^2_+$,} we can easily see that $\overline{{\rm dist}}_T({\cal SDD}_2^*,{\cal S}^2_+)=\frac{n-2}{n}=0$.

Lemmas 1 and 2 are based on the proofs of Theorems 3 and 1 in \cite{blekherman2020}.

\setcounter{theorem}{0}
\begin{lemma}\label{thsddlow}
	For all $n\ge 3$, 
	\begin{align*}
		\overline{{\rm dist}}_T({\cal SDD}_n^*,{\cal S}^n_+)\ge \frac{n-2}{n}.
	\end{align*}
\end{lemma}

\begin{proof}
	Let $I\in\mathbb{S}^n$ be the identity matrix and $e:=(1,\ldots,1)^T\in\mathbb{R}^n$. Given scalars $a,b\ge0$, we define a matrix, 
	\begin{align}
		G(a,b,n):=(a+b)I-aee^T.\label{defg}
	\end{align}
	
	If $G(a,b,n)\in{\cal SDD}_n^*\setminus{\cal S}^n_+$ and Tr$(G(a,b,n))=1$, then by definition $(\ref{deftrdis})$, $\|G(a,b,n)-{\rm P}_{{\cal S}^n_+}(G(a,b,n))\|_F$ gives a lower bound on $\overline{{\rm dist}}_T({\cal SDD}_n^*,{\cal S}^n_+)$. To find a tighter lower bound, we consider the following problem (\ref{p1}) on the parameters $a$ and $b$:
	\begin{subequations}\label{p1}
		\begin{eqnarray}
			&\max_{a,b\ge0}&\|G(a,b,n)-{\rm P}_{{\cal S}^n_+}(G(a,b,n))\|_F\label{eqobj}\\
			&{\rm s.t.}&G(a,b,n)\notin{\cal S}^n_+,\label{eqcons}\\
			&&G(a,b,n)\in{\cal SDD}_n^*,\label{eqcont}\\
			&&{\rm Tr}(G(a,b,n))=1.\label{eqconst}
		\end{eqnarray}
	\end{subequations}
	
	Problem $(\ref{p1})$ can be equivalently written as:
	\begin{subequations}\label{p2}
		\begin{eqnarray}
			&\max_{a,b\ge0}&(n-1)a-b\label{eqobj2}\\
			&{\rm s.t.}&b<(n-1)a,\label{eqcons2}\\
			&&b\ge a,\label{eqcont2}\\
			&&nb=1.\label{eqconst2}
		\end{eqnarray}
	\end{subequations}
	
	To prove the equivalence between (\ref{p1}) and (\ref{p2}), we first show that the constraints (\ref{eqcons}) and (\ref{eqcons2}) are equivalent. Proposition 4 in \cite{blekherman2020} ensures that the eigenvalues of $G(a,b,n)$ are $a+b$ with multiplicity $n-1$ and $b-(n-1)a$ with multiplicity $1$. Note that $a,b\ge0$; hence,
	\begin{align}
		G(a,b,n)\notin{\cal S}^n_+\mbox{ if and only if }b<(n-1)a.\label{eqgs}
	\end{align} 
	
	Next, we verify that (\ref{eqcont}) and (\ref{eqcont2}) are equivalent. It follows from definition $(\ref{def:sdd})$ that $G(a,b,n)\in{\cal SDD}^*_n$ if and only if all the $2\times 2$ submatrices of $G(a,b,n)$ are positive semidefinite. It is obvious from $(\ref{defg})$ that any $2\times2$ submatrix of $G(a,b,n)$ is $G(a,b,2)$. $(\ref{eqgs})$ ensures that $G(a,b,2)\in{\cal S}^2_+$ if and only if $b\ge a$ and we can conclude that
	\begin{align}
		G(a,b,n)\in{\cal SDD}^*_n\mbox{ if and only if }b\ge a.\label{eqsddif}
	\end{align}
	
	The equivalence between (\ref{eqconst}) and (\ref{eqconst2}) comes from the fact that the definition $(\ref{defg})$ implies that
	\begin{align}
		{\rm Tr}(G(a,b,n))=nb.\label{eqsddtr}
	\end{align}
	
	Finally, we show that the objective functions (\ref{eqobj}) and (\ref{eqobj2}) are equivalent. Since (\ref{eqcons}) implies that $G(a,b,n)\notin {\cal S}^n_+$, it is apparent from (\ref{eqgs}) that $b-(n-1)a<0$. Then, $b-(n-1)a$ is the only negative eigenvalue of $G(a,b,n)$, and hence,
	\begin{align}
		\|G(a,b,n)-{\rm P}_{{\cal S}^n_+}(G(a,b,n))\|_F=(n-1)a-b.\label{eqsddobj}
	\end{align}
	
	One can see that problems (\ref{p1}) and (\ref{p2}) are equivalent from (\ref{eqgs}), (\ref{eqsddif}), (\ref{eqsddtr}) and (\ref{eqsddobj}). {Then for any feasible solution $(a,b)$ of (\ref{p2}), $(n-1)a-b$ gives a lower bound on $\overline{{\rm dist}}_T({\cal SDD}_n^*,{\cal S}^n_+)$.} The optimal solution of problem (\ref{p2}) is $\bar{a}=\bar{b}=\frac{1}{n}$; hence, we have
	\begin{align*}
		\overline{\mbox{\rm dist}}_T({\cal SDD}_n^*,{\cal S}^n_+)\ge\|G(\bar{a},\bar{b},n)-{\rm P}_{{\cal S}^n_+}(G(\bar{a},\bar{b},n))\|_F=\frac{n-2}{n}.
	\end{align*}

\end{proof}

\begin{lemma}\label{thmy3}
	For all $n\ge 3$, 
	\begin{align*}
		\overline{{\rm dist}}_T({\cal SDD}_n^*,{\cal S}^n_+)\leq \frac{n-2}{n}.
	\end{align*}
\end{lemma}

\begin{proof}
	If a scalar $U$ satisfies $\|X-{\rm P}_{{\cal S}^n_+}(X)\|_F\leq U$ for every $X\in{\cal SDD}_n^*$ with ${\rm Tr}(X)=1$, then $U$ is an upper bound on $\overline{{\rm dist}}_T({\cal SDD}_n^*,{\cal S}^n_+)$. Below, we find such a scalar $U$.
	
	Let $X$ be a matrix in ${\cal SDD}_n^*$ satisfying ${\rm Tr}(X)=1$. We construct a matrix $\tilde{X}\in{\cal S}^n_+$ and a scalar $\tilde{\alpha}\ge0$ {in a way such that $\|X-\tilde{\alpha}\tilde{X}\|_F\leq U$, which then shows that $\|X-P_{{\cal S}^n_+}(X)\|_F\leq U$.} 
	
	Define a matrix $X^{(i,j)}\in\mathbb{S}^n$ for every $1\leq i<j\leq n$:
	\begin{align}\label{def:xijsdd}
		X^{(i,j)}_{p,q}:=\begin{cases}
			X_{i,i} & ({\rm if}\ p=q=i),\\
			X_{j,j} & ({\rm if}\ p=q=j),\\
			X_{i,j}&({\rm if}\ (p,q)\in\{(i,j),(j,i)\}),\\
			0&({\rm otherwise}).
		\end{cases}
	\end{align}
	Let $C_n^k:=\frac{n!}{(n-k)!k!}$ and let $\tilde{X}=\frac{1}{C_n^2}\sum_{1\leq i<j\leq n}X^{(i,j)}$. Then, $(\ref{def:xijsdd})$ implies that
	\begin{align*}
		\tilde{X}_{i,i}=&\frac{C_n^2-C_{n-1}^2}{C_n^2}X_{i,i}=\frac{2}{n}X_{i,i}\ (i=1,\ldots,n),\\
		\tilde{X}_{i,j}=&\frac{1}{C_n^2}X_{i,j}=\frac{2}{n(n-1)}X_{i,j}\ (1\leq i<j\leq n).
	\end{align*}
	By $(\ref{def:sdd})$, we know that $X^{(i,j)}\in{\cal S}^n_+$ for all $1\leq i<j\leq n$ and hence $\tilde{X}\in{\cal S}^n_+$. 
	
	Let $\alpha\ge0$ be any scalar. Then, 
	\begin{align}
		\|X-\alpha\tilde{X}\|_F=&\sqrt{\sum_{i=1}^n(X-\alpha\tilde{X})_{i,i}^2+\sum_{i\neq j}(X-\alpha\tilde{X})_{i,j}^2}\nonumber\\
		=&\sqrt{ \sum_{i=1}^n(1-\frac{2\alpha}{n})^2X_{i,i}^2+ \sum_{i\neq j} (1-\frac{2\alpha}{n(n-1)})^2X_{i,j}^2 }\nonumber\\
		\leq&\sqrt{ (1-\frac{2\alpha}{n})^2\sum_{i=1}^nX_{i,i}^2+ (1-\frac{2\alpha}{n(n-1)})^2\sum_{i\neq j} X_{i,i}X_{j,j} }\nonumber\\
		=&\sqrt{ (1-\frac{2\alpha}{n})^2\sum_{i=1}^nX_{i,i}^2+ (1-\frac{2\alpha}{n(n-1)})^2 ({\rm Tr}(X)^2-\sum_{i=1}^nX_{i,i}^2) }\nonumber\\
		=&\sqrt{ \left((1-\frac{2\alpha}{n})^2- (1-\frac{2\alpha}{n(n-1)})^2\right)\sum_{i=1}^nX_{i,i}^2+ (1-\frac{2\alpha}{n(n-1)})^2 {\rm Tr}(X)^2 }\nonumber\\
		=&\sqrt{ \left(1-\frac{4\alpha}{n}+\frac{4\alpha^2}{n^2}- (1-\frac{4\alpha}{n(n-1)}+\frac{4\alpha^2}{n^2(n-1)^2})\right)\sum_{i=1}^nX_{i,i}^2+ (1-\frac{2\alpha}{n(n-1)})^2 {\rm Tr}(X)^2 }\nonumber\\
		=&\sqrt{ \left( \frac{4\alpha^2(n-2)}{n(n-1)^2}-\frac{4\alpha(n-2)}{n(n-1)} \right)\sum_{i=1}^nX_{i,i}^2+ (1-\frac{2\alpha}{n(n-1)})^2 {\rm Tr}(X)^2 }.\label{eq:thmy3:1}
	\end{align}
	Note that Tr$(X)=1$ and $\tilde{\alpha}:=n-1\ge0$ satisfies that $\frac{4\tilde{\alpha}^2(n-2)}{n(n-1)^2}-\frac{4\tilde{\alpha}(n-2)}{n(n-1)} =0$. By substituting $\tilde{\alpha}$ into $(\ref{eq:thmy3:1})$, we have
	\begin{align*}
		\|X-\tilde{\alpha}\tilde{X}\|_F\leq&\sqrt{ (1-\frac{2\tilde{\alpha}}{n(n-1)})^2 }=\frac{n-2}{n}.
	\end{align*}
	Since $\tilde{\alpha}\ge0$ and $\tilde{X}\in{\cal S}^n_+$, by letting $U=\frac{n-2}{n}$, we can see that
	\begin{align*}
		\|X-{\rm P}_{{\cal S}^n_+}(X)\|_F\leq \|X-\tilde{\alpha}\tilde{X}\|_F\leq U=\frac{n-2}{n},
	\end{align*}
	and hence,
	\begin{align*}
		\overline{{\rm dist}}_T({\cal SDD}_n^*,{\cal S}^n_+)=\sup_{X\in{\cal SDD}_n^*,{\rm Tr}(X)=1}\|X-{\rm P}_{{\cal S}^n_+}(X)\|_F\leq U= \frac{n-2}{n}.
	\end{align*}
\end{proof}

\subsection{The trace normalized distance between ${\cal DD}_n^*$ and ${\cal S}^n_+$}
In this section, we prove the following theorem:
\begin{theorem}\label{thmy7}
	For all $n\ge 2$, 
	\begin{align*}
		\overline{{\rm dist}}_T({\cal DD}_n^*,{\cal S}^n_+)=\frac{\sqrt{n}-1}{2}.
	\end{align*}
\end{theorem}

The idea behind Theorem \ref{thmy7} is as follows. Define
\begin{align}
	{\cal DDT}^*_n:={\cal DD}^*_n\cap\{X\in\mathbb{S}^n\mid{\rm Tr}(X)=1\}.\label{def:ddt}
\end{align}
Definition $(\ref{deftrdis})$ ensures that
\begin{align*}
	\overline{{\rm dist}}_T({\cal DD}_n^*,{\cal S}^n_+)=\max_{X\in{\cal DDT}^*_n}\|X-{\rm P}_{{\cal S}^n_+}(X)\|_F.
\end{align*}
Note that $\|X-{\rm P}_{{\cal S}^n_+}(X)\|_F$ is continuous and convex on $\mathbb{S}^n$ and ${\cal DDT}^*_n$ is closed, bounded, and convex. The Bauer maximum principle \cite{bauer1958} states that any continuous convex function defined on a compact convex set in $\mathbb{R}^n$ attains its maximum at some extreme point of the set. As a corollary, we have the following:

\setcounter{theorem}{0}
\begin{corollary}\label{colo:ex}
	$\max_{X\in{\cal DDT}^*_n}\|X-{\rm P}_{{\cal S}^n_+}(X)\|_F$ attains its maximum at some extreme point of ${\cal DDT}^*_n$.
\end{corollary}

Proposition \ref{tm:ex} shows that every extreme point of ${\cal DDT}^*_n$ has a special structure.  Lemma \ref{thmy8} uses this special structure to show that the distance from each extreme point of ${\cal DDT}^*_n$ to ${\cal S}^n_+$ is the same. The exact distance is also given in Lemma \ref{thmy8}. Theorem \ref{thmy7} follows directly from Corollary \ref{colo:ex} and Lemma \ref{thmy8}.

\setcounter{theorem}{0}
\begin{proposition}\label{tm:ex}
	For $n\ge2$, let $X$ be an extreme point of ${\cal DDT}^*_n$. There exists an integer $q$ satisfying $1\leq q\leq n$ such that
	\begin{align}
		X_{i,j}=\begin{cases}
			1 & (\mbox{ if}\ i=j=q),\\
			\frac{1}{2} \mbox{ or }-\frac{1}{2} &(\mbox{ if either}\ i=q \mbox{ or }j=q), \\
			0&(\mbox{ otherwise}).
		\end{cases}\label{xijpro}
	\end{align}
\end{proposition}

\begin{proof}
	Let $X\in{\cal DDT}^*_n$. By (\ref{def:dd}) and (\ref{def:ddt}), we see that for every $i=1,\ldots,n$, 
	\begin{align*}
		X_{i,i}\ge0,
	\end{align*}
	and for every $1\leq i<j\leq n$, 
	\begin{align}
		&X_{i,i}+X_{j,j}+2X_{i,j}\ge0, \label{eq:ex:2}\\
		&X_{i,i}+X_{j,j}-2X_{i,j}\ge0. \label{eq:ex:3} 
	\end{align}
	Thus, the set ${\cal DDT}^*_n$ can be written as 
	\begin{align}
		{\cal DDT}^*_n=\{X\in\mathbb{S}^n\mid& {\rm Tr}(X)=1,\nonumber\\
		&X_{i,i}\ge0\ ( i=1,\ldots,n), \label{eq:ex:1}\\
		&X_{i,i}+X_{j,j}+2X_{i,j}\ge0\ (1\leq i<j\leq n), \label{eq:ex:4}\\
		&X_{i,i}+X_{j,j}-2X_{i,j}\ge0\ (1\leq i<j\leq n) \label{eq:ex:5}\}.
	\end{align}
	
	Let $\bar{X}$ be an extreme point of ${\cal DDT}^*_n$ and let $N(X)$ be the number of linearly independent inequalities in (\ref{eq:ex:1}), (\ref{eq:ex:4}) and (\ref{eq:ex:5}) that are active (i.e., the equalities hold) at $X\in{\cal DDT}^*_n$. From a characterization of the extreme points of a polyhedron (see, e.g., Theorem 5.7, \cite{schrijver2003}), we know that
	\begin{align}
		N(\bar{X})=\frac{n(n+1)}{2}-1.\label{eq:numext}
	\end{align}
	Below, we prove that $\bar{X}$ satisfies (\ref{xijpro}) by observing the active inequalities at $\bar{X}$. 
	
	It follows from Tr$(\bar{X})=1$ that $\bar{X}$ has at least one nonzero diagonal element. This implies that the number of active inequalities in $(\ref{eq:ex:1})$ at $\bar{X}$ is at most $n-1$. Suppose that $n-k$ inequalities in $(\ref{eq:ex:1})$ are active at $\bar{X}$, where $k$ is an integer and $1\leq k\leq n$. Below, we show that $k\neq n$ by contradiction.
	
	Assume that $k=n$. Then we have $\bar{X}_{i,i}>0$ for each $1\leq i \leq n$. At most one of $(\ref{eq:ex:2})$ and $(\ref{eq:ex:3})$ can be active at $\bar{X}$ for each $1\leq i<j\leq n$. In fact, suppose that $(\ref{eq:ex:2})$ and $(\ref{eq:ex:3})$ are simultaneously active for some $1\leq i<j\leq n$:
	\begin{align*}
		\bar{X}_{i,i}+\bar{X}_{j,j}+2\bar{X}_{i,j}=0,\ \bar{X}_{i,i}+\bar{X}_{j,j}-2\bar{X}_{i,j}=0.
	\end{align*}
	Then, $\bar{X}_{i,j}=\bar{X}_{i,i}+\bar{X}_{j,j}=0$ and since $\bar{X}_{i,i},\bar{X}_{j,j}\ge0$, we obtain $\bar{X}_{i,i}=\bar{X}_{j,j}=0$, which is a contradiction to the assumption $\bar{X}_{i,i},\bar{X}_{j,j}>0$. This implies that $N(\bar{X})$ is at most $\frac{n(n-1)}{2}$, which is strictly less than the number $\frac{n(n+1)}{2}-1$ in (\ref{eq:numext}). This contradiction implies that $k\neq n$.
	
	Since we have shown that $1\leq k\leq n-1$, there exists a permutation matrix $P\in\mathbb{R}^{n\times n}$ such that the matrix $X^*:=P\bar{X}P^T$ satisfies 
	\begin{align}
		&X^*_{i,i}=0\hspace{3mm} (1\leq i\leq n-k),\label{eq:ex:6} \\
		&X^*_{i,i}>0\hspace{3mm} (n-k+1\leq i \leq n).\nonumber
	\end{align}
	Note that $X^*\in{\cal DDT}^*_n$ and $N(X^*)=\frac{n(n+1)}{2}-1$. Below, we show that $X^*$ satisfies (\ref{xijpro}) by observing the active inequalities at $X^*$ instead of $\bar{X}$. 
	
	Next, we show that $k=1$; i.e., exactly $n-1$ inequalities in $(\ref{eq:ex:1})$ are active at $X^*$. It follows from (\ref{eq:ex:4}), (\ref{eq:ex:5}) and (\ref{eq:ex:6}) that $X^*_{i,j}=0$ for each $1\leq i<j\leq n-k$. This implies that all inequalities (\ref{eq:ex:2}) and (\ref{eq:ex:3}) with $1\leq i<j\leq n-k$ at $X^*$ are active. For each pair of $(i,j)$ where $X_{j,j}>0$ and $1\leq i< j$, one can show again by contradiction that at most one of $(\ref{eq:ex:2})$ and $(\ref{eq:ex:3})$ can be active at $X^*$. Consider the case when the number of active inequalities at $X^*$ attains its maximum; i.e., exactly one of $(\ref{eq:ex:2})$ and $(\ref{eq:ex:3})$ is active at $X^*$ for each pair of $(i,j)$, where $n-k+1\leq j\leq n$ and $1\leq i<j$. The following system,
	\begin{align*}
		\begin{cases}
			0=X^*_{i,i}\hspace{3mm} (1\leq i\leq n-k), \\
			0=X^*_{i,i}+X^*_{j,j}+2X^*_{i,j}\hspace{3mm} (1\leq i<j \leq n-k), \\
			0=X^*_{i,i}+X^*_{j,j}-2X^*_{i,j}\hspace{3mm} (1\leq i<j \leq n-k), \\
			\mbox{either (\ref{eq:ex:2}) or (\ref{eq:ex:3}) is active at }X^* \hspace{3mm} (n-k+1\leq j\leq n,\ 1\leq i<j),
		\end{cases}
	\end{align*}
	includes exactly $(n-k)+\frac{n(n-1)}{2}=\frac{n(n+1)}{2}-k$ linearly independent active inequalities. This implies that $N(X^*)\leq\frac{n(n+1)}{2}-k$. By (\ref{eq:numext}), we know that $k=1$ and $\frac{n(n+1)}{2}-1$ in (\ref{eq:numext}) is attained only if the number of active inequalities in $(\ref{eq:ex:4})$ and $(\ref{eq:ex:5})$ attains its maximum. 
	
	$k=1$ implies that $X^*_{n,n}=1$ and $X^*_{i,i}=0$ for each $1\leq i\leq n-1$; and hence, $X^*_{i,j}=0$ for each $1\leq i<j\leq n-1$. Since the number of active inequalities in $(\ref{eq:ex:4})$ and $(\ref{eq:ex:5})$ attains its maximum, we know that either (\ref{eq:ex:2}) or (\ref{eq:ex:3}) is active at $X^*$ for each $(i,j)$ satisfying $j=n$ and $1\leq i<j$, which implies that $X^*_{i,n}\in\{\frac{1}{2},-\frac{1}{2}\}$ for each $1\leq i<n$. 
	
	Finally, by applying the permutation $\bar{X}=P^TX^*P$, we know that there exists an integer $q$ satisfying $1\leq q\leq n$ for which $\bar{X}$ satisfies (\ref{xijpro}).
	
\end{proof}

\setcounter{theorem}{2}
\begin{lemma}\label{thmy8}
	For $n\ge2$, let $X$ be an extreme point of ${\cal DDT}_n^*$. There exist scalars $\alpha_1,\ldots,\alpha_{n-1}\in\{\frac{1}{2},-\frac{1}{2}\}$ such that the following matrix,
	{\begin{align}
			X^*:=\left(
			\begin{array}{cccc}
				0 & & & \alpha_1 \\
				& \ddots & &\vdots \\
				&& 0 &\alpha_{n-1} \\
				\alpha_1& \ldots &\alpha_{n-1}&1
			\end{array}
			\right)\label{eq:max}
	\end{align}}
	satisfies
	\begin{align*}
		\|X-{\rm P}_{{\cal S}^n_+}(X)\|_F=\|X^*-{\rm P}_{{\cal S}^n_+}(X^*)\|_F=\frac{\sqrt{n}-1}{2}.
	\end{align*}
\end{lemma}

\begin{proof}
	Let $X$ be an extreme point of ${\cal DDT}_n^*$. By Proposition \ref{tm:ex}, there exists an integer $q$ such that $1\leq q\leq n$ for which $X$ satisfies (\ref{xijpro}). Note that $X$ only has one nonzero diagonal element $X_{q,q}=1$. Let $P\in\mathbb{R}^{n\times n}$ be a permutation matrix such that $(PXP^T)_{n,n}=1$. It is easy to see that there are scalars $\alpha_1,\ldots,\alpha_{n-1}\in\{\frac{1}{2},-\frac{1}{2}\}$ such that the matrix $X^*$ defined in $(\ref{eq:max})$ satisfies $X^*=PXP^T$. Since the permutation matrix $P$ is orthogonal, we see that $Y\in{\cal S}^n_+$ if and only if $PYP^T\in{\cal S}^n_+$ for any $Y\in\mathbb{S}^n$. This fact implies that
	\begin{align}
		\|X-{\rm P}_{{\cal S}^n_+}(X)\|_F=&\inf_{Y\in{\cal S}^n_+}\|X-Y\|_F\nonumber\\
		=&\inf_{Y\in{\cal S}^n_+}\|PXP^T-PYP^T\|_F\nonumber\\
		=&\inf_{PYP^T\in{\cal S}^n_+}\|X^*-PYP^T\|_F\nonumber\\
		=&\|X^*-{\rm P}_{{\cal S}^n_+}(X^*)\|_F.\label{eq:diseq}
	\end{align}
	
	By solving the eigenvalue equation $ 0=|\lambda I-X^*|$ with respect to the scalar $\lambda$, we obtain that:
	\begin{description}
		\item[1.] If $n=2$, the eigenvalues of $X^*$ are $\frac{1+\sqrt{n}}{2}$ with multiplicity $1$ and $\frac{1-\sqrt{n}}{2}$ with multiplicity $1$.
		\item[2.] If $n\ge3$,  the eigenvalues of $X^*$ are $\frac{1+\sqrt{n}}{2}$ with multiplicity $1$, $\frac{1-\sqrt{n}}{2}$ with multiplicity $1$ and $0$ with multiplicity $n-2$.
	\end{description}
	From these observations, for every $n\ge2$, $X^*$ has only one negative eigenvalue $\lambda_{\rm min}:=\frac{1-\sqrt{n}}{2}$; hence,
	\begin{align}
		\|X^*-{\rm P}_{{\cal S}^n_+}(X^*)\|_F=\sqrt{\lambda_{\rm min}^2}=\frac{\sqrt{n}-1}{2}.\label{eqeig}
	\end{align}
	
	We conclude from (\ref{eq:diseq}) and (\ref{eqeig}) that
	\begin{align*}
		\|X-{\rm P}_{{\cal S}^n_+}(X)\|_F=\|X^*-{\rm P}_{{\cal S}^n_+}(X^*)\|_F=\frac{\sqrt{n}-1}{2}.
	\end{align*}
	
\end{proof}

\section{Concluding remarks}
\label{sec:3}
We showed that the norm normalized distance $\overline{{\rm dist}}_F({\cal S},{\cal S}^n_+)$ has the same value whenever ${\cal SDD}_n^* \subseteq {\cal S} \subseteq {\cal DD}_n^*$, since $\overline{{\rm dist}}_F({\cal DD}_n^*,{\cal S}^n_+)=\overline{{\rm dist}}_F({\cal SDD}_n^*,{\cal S}^n_+)$ holds. This implies that the norm normalized distance is not a sufficient measure to evaluate these approximations. Moreover, as a new measure to compensate for the weakness of that distance, we proposed a new distance, the trace normalized distance $\overline{{\rm dist}}_T({\cal S},{\cal S}^n_+)$. Using this new measure, we proved that $\overline{{\rm dist}}_T({\cal DD}_n^*,{\cal S}^n_+)$ and $\overline{{\rm dist}}_T({\cal SDD}_n^*,{\cal S}^n_+)$ are different, i.e., $\overline{{\rm dist}}_T({\cal DD}_n^*,{\cal S}^n_+)=\frac{\sqrt{n}-1}{2}$ and $\overline{{\rm dist}}_T({\cal SDD}_n^*,{\cal S}^n_+)= \frac{n-2}{n}$.

In \cite{wang2021}, the authors proposed a class of polyhedral approximations of the semidefinite cone, denoted as ${\cal SDB}_n^*$. The experimental results on cutting-plane methods, where ${\cal SDB}_n^*$ is used as an approximation of ${\cal S}^n_+$ for solving SDP instances are promising. It is an interesting but also challenging issue to analyze the value of $\overline{{\rm dist}}_T({\cal SDB}_n^*,{\cal S}^n_+)$.{ To evaluate approximations of the semidefinite cone, one may also study general properties of the set of eigenvalues that are attained by matrices in the approximation sets. For example, Kozhasov \cite{kozhasov} proved that the set of eigenvalue vectors of matrices in ${\cal S}^{4,2}={\cal SDD}^*_4$ is not convex. It is also an interesting and challenging direction to study the set of eigenvalues that are attained by matrices in ${\cal DD}^*_n$ and ${\cal SDD}^*_n$ respectively.}

\section{Acknowledgments}
This research was supported by the Japan Society for the Promotion of Science through a Grant-in-Aid for Scientific Research ((B)19H02373) from the Ministry of Education, Culture, Sports, Science and Technology of Japan. {The authors would like to sincerely thank the anonymous reviewers for their thoughtful and valuable comments which have significantly improved the paper.}

\section*{Appendix}
\begin{appendix}
	
	\section{Proof of Theorem \ref{thmy2}}\label{app1}
	
	\begin{proof}
		We prove the following inequalities:
		\begin{align}
			\frac{n-2}{n}\leq \overline{\mbox{dist}}_F({\cal SDD}_n^*,{\cal S}^n_+)\leq \overline{\mbox{dist}}_F({\cal DD}_n^*,{\cal S}^n_+)\leq \frac{n-2}{n}.\label{eq:ineq}
		\end{align}
		
		The relation ${\cal S}^{n,2}={\cal SDD}_n^*$ in $(\ref{setrela})$ implies that $\overline{\mbox{dist}}_F({\cal SDD}_n^*,{\cal S}^n_+)=\overline{\mbox{dist}}_F({\cal S}^{n,2},{\cal S}^n_+)$. By Theorem 3 in \cite{blekherman2020}, we know that $\overline{\mbox{dist}}_F({\cal S}^{n,k},{\cal S}^n_+)\ge\frac{n-k}{\sqrt{(k-1)^2n+n(n-1)}}$ and hence that,
		\begin{align}
			\overline{\mbox{dist}}_F({\cal SDD}_n^*,{\cal S}^n_+)=\overline{\mbox{dist}}_F({\cal S}^{n,2},{\cal S}^n_+)\ge\frac{n-2}{\sqrt{(2-1)^2n+n(n-1)}}=\frac{n-2}{n}.\label{eq:low}
		\end{align}
		
		The relation ${\cal SDD}_n^*\subseteq{\cal DD}_n^*$ in $(\ref{setrela})$ ensures that 
		\begin{align}
			\overline{\mbox{dist}}_F({\cal SDD}_n^*,{\cal S}^n_+)\leq\overline{\mbox{dist}}_F({\cal DD}_n^*,{\cal S}^n_+).\label{eq:high1}
		\end{align}
		Next, we prove that $\overline{\rm dist}_F({\cal DD}_n^*,{\cal S}^n_+)\leq \frac{n-2}{n}$ with the following idea. If a scalar $U$ satisfies $\|X-{\rm P}_{{\cal S}^n_+}(X)\|_F\leq U$ for every $X\in{\cal DD}_n^*$ with $\|X\|_F=1$, then $U$ is an upper bound on $\overline{{\rm dist}}_F({\cal DD}_n^*,{\cal S}^n_+)$. We can find such a scalar $U$ by constructing a matrix $\tilde{X}\in{\cal S}^n_+$ and a scalar $\tilde{\alpha}\ge0$ for every $X\in{\cal DD}_n^*$ with $\|X\|_F=1$ in such a way that $\|X-{\rm P}_{{\cal S}^n_+}(X)\|_F\leq\|X-\tilde{\alpha}\tilde{X}\|_F$. 
		
		Let $X$ be a matrix in ${\cal DD}_n^*$ satisfying $\|X\|_F=1$. Define a matrix $X^{(i,j)}\in\mathbb{S}^n$ for every $1\leq i<j\leq n$ :
		\begin{align}\label{def:xij}
			X^{(i,j)}_{p,q}:=\begin{cases}
				\frac{X_{i,i}+X_{j,j}}{2} & ({\rm if}\ p=q\in\{i,j\}),\\
				X_{i,j}&({\rm if}\ (p,q)\in\{(i,j),(j,i)\}),\\
				0&({\rm otherwise}).
			\end{cases}
		\end{align}
		Let $\bar{X}:={\frac{2}{n(n-1)}}\sum_{1\leq i<j\leq n}X^{(i,j)}$. By definitions $(\ref{def:dd})$ and $(\ref{def:xij})$, one can verify that $X^{(i,j)}\in{\cal S}^n_+$ for all $1\leq i<j\leq n$ and hence that $\bar{X}\in{\cal S}^n_+$. Let $\alpha$ be a scalar satisfying $\alpha\ge\frac{2n(n-1)}{3n-4}>0$. For all $1\leq i<j\leq n$, we can obtain from $(\ref{def:xij})$ that $\bar{X}_{i,j}=\bar{X}_{j,i}={\frac{2}{n(n-1)}}X_{i,j}$ and hence that
		\begin{align}
			\sum_{i\neq j}(X_{i,j}-\alpha\bar{X}_{i,j})^2=\sum_{i\neq j}(1-{\frac{2\alpha}{n(n-1)}})^2X_{i,j}^2.\label{eq:thmyf:7}
		\end{align}
		For all $i=1,\ldots,n$, $(\ref{def:xij})$ implies that $\bar{X}_{i,i}={\frac{2}{n(n-1)}}\left(\frac{n-2}{2}X_{i,i}+\frac{1}{2}{\rm Tr}(X)\right)$ and hence that
		\begin{align}
			\sum_{i=1}^n(X_{i,i}-\alpha\bar{X}_{i,i})^2=&\sum_{i=1}^n((1-\frac{\alpha(n-2)}{n(n-1)})X_{i,i}-\frac{\alpha}{n(n-1)}{\rm Tr}(X))^2\nonumber\\
			=&\sum_{i=1}^n((1-\frac{\alpha(n-2)}{n(n-1)})^2X_{i,i}^2-2(1-\frac{\alpha(n-2)}{n(n-1)})X_{i,i}\frac{\alpha}{n(n-1)}{\rm Tr}(X)\nonumber\\
			&+\frac{\alpha^2}{n^2(n-1)^2}{\rm Tr}(X)^2)\nonumber\\
			=&(1-\frac{\alpha(n-2)}{n(n-1)})^2\sum_{i=1}^nX_{i,i}^2-({\frac{2\alpha}{n(n-1)}-\frac{2\alpha^2(n-2)}{n^2(n-1)^2}}){\rm Tr}(X)\sum_{i=1}^nX_{i,i}\nonumber\\
			&+\frac{\alpha^2n}{n^2(n-1)^2}{\rm Tr}(X)^2\nonumber\\
			=&(1-\frac{\alpha(n-2)}{{n(n-1)}})^2\sum_{i=1}^nX_{i,i}^2+(\frac{\alpha^2(3n-4)}{n^2(n-1)^2}-{\frac{2\alpha}{n(n-1)}}){\rm Tr}(X)^2.\label{eq:eqtr1}
		\end{align}
		The assumption $\alpha\ge\frac{2n(n-1)}{3n-4}$ ensures that $\frac{\alpha^2(3n-4)}{n^2(n-1)^2}-{\frac{2\alpha}{n(n-1)}}\ge0$. One can verify that ${\rm Tr}(X)^2\leq n\sum_{i=1}^nX_{i,i}^2$ by using the Cauchy-Schwarz inequality. Then, it follows from $(\ref{eq:eqtr1})$ that
		\begin{align}
			\sum_{i=1}^n(X_{i,i}-\alpha\bar{X}_{i,i})^2\leq&(1-\frac{\alpha(n-2)}{n(n-1)})^2\sum_{i=1}^nX_{i,i}^2+(\frac{\alpha^2(3n-4)}{n^2(n-1)^2}-{\frac{2\alpha}{n(n-1)}})n\sum_{i=1}^nX_{i,i}^2\nonumber\\
			=&\left((1-\frac{\alpha(n-2)}{n(n-1)})^2 +\frac{\alpha^2n(3n-4)}{n^2(n-1)^2}-{\frac{2\alpha n}{n(n-1)}} \right)\sum_{i=1}^nX_{i,i}^2\nonumber\\
			=&\left(1-\frac{2\alpha(n-2)}{n(n-1)}+\frac{\alpha^2(n-2)^2}{n^2(n-1)^2} +\frac{\alpha^2n(3n-4)}{n^2(n-1)^2}-{\frac{2\alpha n}{n(n-1)}} \right)\sum_{i=1}^nX_{i,i}^2\nonumber\\
			=&\left(1-{\frac{2\alpha(n-2+n)}{n(n-1)}}+\frac{\alpha^2(n^2-4n+4+3n^2-4n)}{n^2(n-1)^2}\right)\sum_{i=1}^nX_{i,i}^2\nonumber\\
			=&\left(1-{\frac{4\alpha(n-1)}{n(n-1)}}+\frac{4\alpha^2(n-1)^2}{n^2(n-1)^2}\right)\sum_{i=1}^nX_{i,i}^2\nonumber\\
			=&\left(1-{\frac{2\alpha}{n}}\right)^2\sum_{i=1}^nX_{i,i}^2.\label{eq:thmyf:8}
		\end{align}
		
		Combining $(\ref{eq:thmyf:7})$ and $(\ref{eq:thmyf:8})$ gives
		\begin{align}
			\|X-\alpha\bar{X}\|_F\leq\sqrt{\sum_{i\neq j}(1-{\frac{2\alpha}{n(n-1)}})^2X_{i,j}^2+ \left(1-{\frac{2\alpha}{n}}\right)^2\sum_{i=1}^nX_{i,i}^2}.\label{eq:eqtr2}
		\end{align}
		Note that $\bar{\alpha}:=n-1$ satisfies $\bar{\alpha}\ge\frac{2n(n-1)}{3n-4}$ when $n\ge4$, and the coefficients in (\ref{eq:eqtr2}) satisfy
		\begin{align*}
			1-{\frac{2\bar{\alpha}}{n(n-1)}}=-(1-{\frac{2\bar{\alpha}}{n}})=\frac{n-2}{n}.
		\end{align*}
		Since $\|X\|_F=1$, by substituting $\bar{\alpha}$ into $(\ref{eq:eqtr2})$, we have 
		\begin{align*}
			\|X-\bar{\alpha}\bar{X}\|_F\leq&\sqrt{\left(\frac{n-2}{n}\right)^2\sum_{i\neq j}X_{i,j}^2+\left(\frac{n-2}{n}\right)^2\sum_{i=1}^nX_{i,i}^2}\\
			=&\frac{n-2}{n}\|X\|_F^2\\
			=&\frac{n-2}{n}.
		\end{align*}
		Because $\bar{X}\in{\cal S}^n_+$ and $\bar{\alpha}\ge0$, by letting $U=\frac{n-2}{n}$, we have 
		\begin{align*}
			\|X-{\rm P}_{{\cal S}^n_+}(X)\|_F\leq \|X-\bar{\alpha}\bar{X}\|_F\leq U=\frac{n-2}{n}
		\end{align*}
		and hence,
		\begin{align}
			\overline{\rm dist}_F({\cal DD}_n^*,{\cal S}^n_+)=\sup_{X\in{\cal DD}_n^*,\|X\|_F=1}\|X-{\rm P}_{{\cal S}^n_+}(X)\|_F\leq U=\frac{n-2}{n}.\label{eq:high2}
		\end{align}
		
		$(\ref{eq:low})$, $(\ref{eq:high1})$ and $(\ref{eq:high2})$ imply that $(\ref{eq:ineq})$ holds, which proves this theorem.

	\end{proof}
	
\end{appendix}
\end{document}